\documentclass[reqno,12pt]{amsart}

\usepackage{amsmath,amssymb}
\usepackage{amsthm}
\usepackage[left=3.5cm,top=3.5cm,right=3.5cm,bottom=3.5cm]{geometry}
\usepackage{epsfig}

\usepackage{graphicx}
\usepackage{subfigure}
\usepackage{caption}
\usepackage{subcaption}
\usepackage{xcolor}

\definecolor{BLUE_D}{HTML}{29ABCA}

\newcommand{\cool}[1]{\mathrm{CL}(#1)}
\newcommand{\diam}{\mathrm{diam}}
\newcommand{\border}{N}

\newtheorem{theorem}{Theorem}
\newtheorem{lemma}[theorem]{Lemma}
\newtheorem{corollary}[theorem]{Corollary}

\newcommand{\ceil}[1]{\left\lceil#1\right\rceil}
\usepackage{wrapfig}
\usepackage{amssymb}
\usepackage{amsmath}
\usepackage{tikz}
\usetikzlibrary{arrows.meta}
\usetikzlibrary{quotes,angles,positioning}

\title{How to cool a graph}

\author[A.\ Bonato]{Anthony Bonato}
\author[T.G.\ Marbach]{Trent G.\ Marbach}
\author[H.\ Milne]{Holden Milne}
\author[T.\ Mishura]{Teddy Mishura}
\address[A1,A2,A3,A4]{Toronto Metropolitan University, Toronto, Canada}

\email[A1]{(A1) abonato@torontomu.ca}
\email[A2]{(A2) trent.marbach@torontomu.ca}
\email[A3]{(A3) holden.milne@torontomu.ca}
\email[A4]{(A4) tmishura@torontomu.ca}

\begin{document}

\keywords{localization number, limited visibility, pursuit-evasion games, isoperimetric inequalities, graphs}
\subjclass{05C57,05C12}

\maketitle

\begin{abstract}

We introduce a new graph parameter called the cooling number, inspired by the spread of influence in networks and its predecessor, the burning number. The cooling number measures the speed of a slow-moving contagion in a graph; the lower the cooling
number, the faster the contagion spreads. We provide tight bounds on the cooling number via a graph's order and diameter. Using isoperimetric results, we derive the cooling number of Cartesian grids. The cooling number is studied in graphs
generated by the Iterated Local Transitivity model for social networks. We conclude with open problems.

\end{abstract}

\section{Introduction}

The spread of influence has been studied since the early days of modern network science; see \cite{DR,KJT1,KJT2,RD}. From the spread of memes and disinformation in social networks like X, TikTok, and Instagram to the spread of viruses such as COVID-19 and influenza in human contact networks, the spread of influence is a central topic. Common features of most influence spreading models include nodes infecting their neighbors, with the spread governed by various deterministic or stochastic rules. The simplest form of influence spreading is for a node to infect all of its neighbors.

Inspired by a desire for a simplified, deterministic model for influence spreading and by pursuit-evasion games and processes such as Firefighter played on graphs, in \cite{BRJ1,BRJ2}, burning was introduced. Unknown to the authors then, a similar problem was studied in the context of hypercubes much earlier by Noga Alon \cite{A}. In burning, nodes are either burning or not burning, with all nodes initially labeled as not burning. Burning plays out in discrete rounds or time-steps; we choose one node to burn in the first round. In subsequent rounds, neighbors of burning nodes themselves become burning, and in each round, we choose an additional source of burning from the nodes that are not burned (if such a node exists).  Those sources taken in order of selection form a \textit{burning sequence}. Burning ends on a graph when all nodes are burning, and the minimum length of a burning sequence is the \emph{burning number} of $G$, denoted $b(G)$.

Much is now known about the burning number of a graph. For example, burning a graph reduces to burning a spanning tree. The burning problem is \textbf{NP}-complete on many graph families, such as the disjoint union of paths, caterpillars with maximum degree three, and spiders, which are trees where there is a unique node of degree at least three. For a connected graph $G$ of order $n$, it is known that $b(G) \le \sqrt{4n/3} + 1,$ and it is conjectured that the bound can be improved to $\lceil\sqrt{n}\rceil$ (which is the burning number of a path of order $n$). For more on burning, see the survey \cite{AB1} and book \cite{AB2}.

Burning models explosive spread in networks, as in the case of a rapidly spreading meme on social media. However, as is the case for viral outbreaks such as COVID-19, while the spread may happen rapidly based on close contact, it can be mitigated by social distancing and other measures such as ventilation and vaccination. In a certain sense, a dual problem to burning is how to slow an infection as much as possible, as one would attempt to do in a pandemic. For this, we introduce a new contact process called cooling. Cooling spreads analogous to burning, with cooled neighbors spreading their infection to neighboring nodes, and a new cooling source is chosen in each round. However, where burning seeks to minimize the number of rounds to burn all nodes, cooling seeks to \emph{maximize} the number of rounds to cool all nodes.

More formally, given a finite, simple, undirected graph G, the cooling process on $G$ is a discrete-time process. Nodes may be either \emph{uncooled} or \emph{cooled} throughout the process. Initially, in the first round, all nodes are uncooled. At each round $t \ge 1,$ one new uncooled node is chosen to cool if such a node is available. We call such a chosen node a \emph{source}. If a node is \emph{cooled}, then it remains in that state until the end of the process. Once a node is cooled in round $t,$ in round $t + 1$, its uncooled neighbors become cooled. The process ends in a given round when all nodes of $G$ are cooled. Sources are chosen in each round for which they are available.

We define the \emph{cooling number} of $G$, written $\mathrm{CL}(G)$, to be the maximum number of rounds for the cooling process to end. A \emph{cooling sequence} is the set of sources taken in order during cooling. We have that $b(G) \le \mathrm{CL}(G)$, and in some cases, the cooling number is much larger than the burning number. Note that while a choice of sources that burns the graph gives an upper bound to the burning number, a choice of sources that cools the graph gives a lower bound to the cooling number.

For all graphs with diameter at most two, we have that $b(G) = \mathrm{CL}(G)$.
Consider a cycle $C_8$ with eight nodes for an elementary example of cooling.
By symmetry, we may choose any node as the initial source. See Figure~\ref{fig:C8_cooling_example} for a cooling sequence of length four. There is no smaller cooling sequence, and so $\mathrm{CL}(C_8)=4$.

\begin{figure}
    \centering
\includegraphics[scale=0.2]{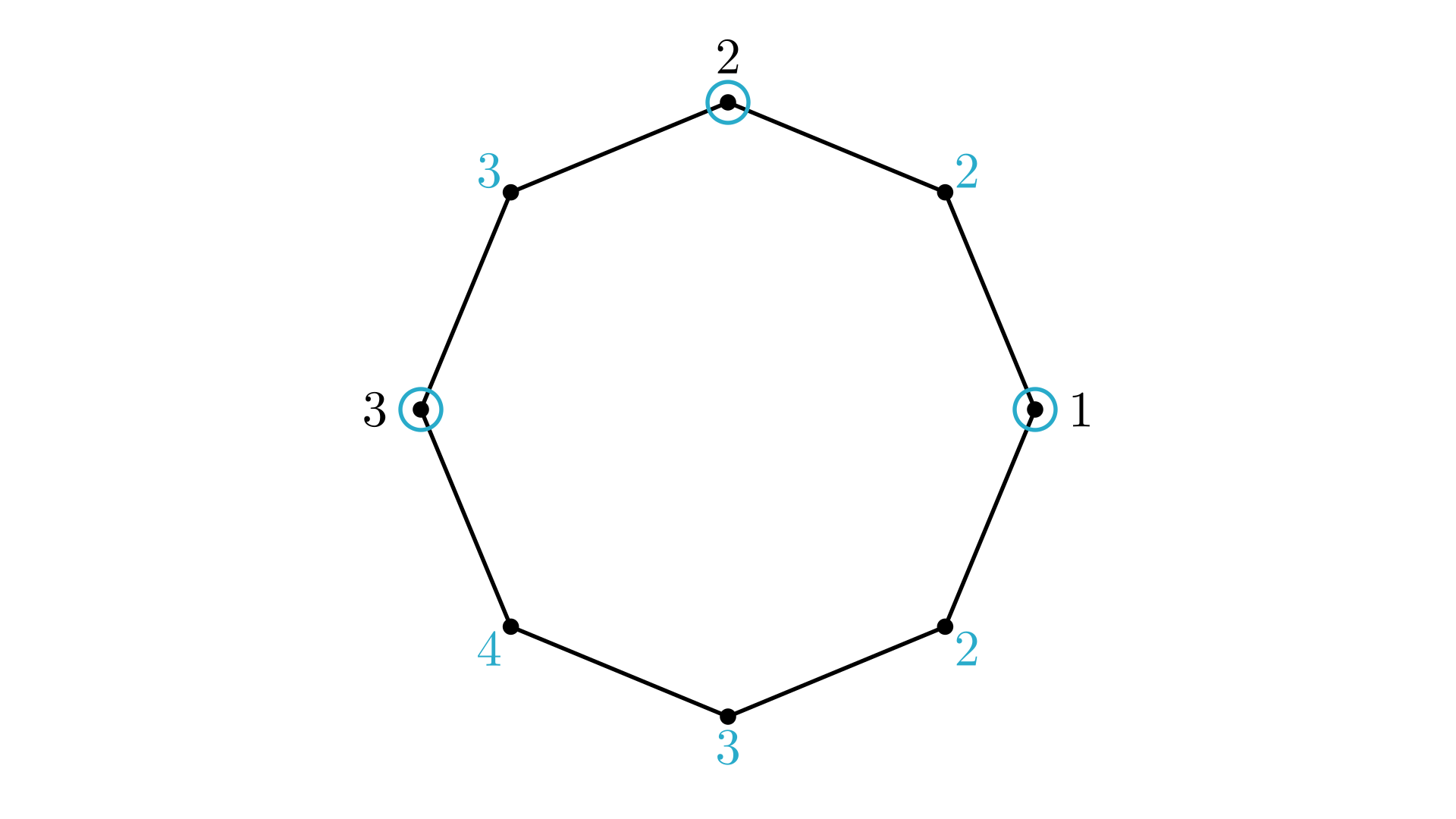}
    \caption{An example of cooling on the cycle of length $8$. Black labels indicate the nodes of the cooling sequence in increasing order. Blue labels indicate the round that the corresponding node was cooled.}
    \label{fig:C8_cooling_example}
\end{figure}

The present paper aims to introduce cooling, provide bounds, and consider its value on various graph families. In Section~2, we will discuss cooling's relation to well-known graph parameters that provide various bounds on the cooling number.
Using isoperimetric properties, we determine the cooling number of Cartesian grid graphs in the next section.
The Iterated Local Transitivity ($\mathrm{ILT}$) model, introduced in \cite{ilt} and further studied in \cite{hyper,tourp,ilm,directed}, simulates structural properties in complex networks emerging from transitivity.
The $\mathrm{ILT}$ model simulates many properties of social networks.
For example, as shown in \cite{ilt}, graphs generated by the model densify over time, have small diameter, high local clustering, and exhibit bad spectral expansion. We derive results for cooling $\mathrm{ILT}$ graphs in Section~4 and prove that the cooling number of $\mathrm{ILT}$ graphs is dependent on the cooling number after two time-steps of the model. We finish with open problems on the cooling number.

All graphs we consider are finite, simple, and undirected. We only consider connected graphs unless otherwise stated. For further background on graph theory, see \cite{west}.

\section{Bounds on the cooling number}

As a warm-up, we consider bounds on the cooling number in terms of various graph parameters. We apply these to derive the cooling number of various graph families, such as paths, cycles, and certain caterpillars. The following elementary theorem bounds the cooling number via a graph's order.

\begin{theorem} \label{thm:upper_n_div_2}
For a graph $G$ on $n$ nodes, we have that \[\cool{G}\leq \Big\lceil \frac{n+1}{2}\Big\rceil.\]
\end{theorem}
\begin{proof}
One uncooled node is cooled from the cooling sequence during each round, except during the last round if all nodes are cooled.
The cooling will spread to at least one additional uncooled node per round, except for the first round.
This implies that there are $$\Big\lceil \frac{n+1}{2}\Big\rceil + \Big\lceil \frac{n+1}{2}\Big\rceil-1 \geq n$$ cooled nodes at the end of round $\lceil \frac{n+1}{2}\rceil$, and the result follows.
\end{proof}

We next bound the cooling number by the diameter.

\begin{theorem} \label{thm:upper_diam_cooling}
For a graph $G$, we have that \[\Big\lceil\frac{\diam(G)+2}{2}\Big\rceil \le \cool{G}\leq \diam(G)+1.\]
\end{theorem}
\begin{proof}
For the upper bound, if some node $v$ is cooled during the first round, then all nodes of distance at most $i$ from $v$ will be cooled by the end of round $i+1$.  All nodes have distance at most $\diam(G)$ from $v$, so all nodes will be cooled by the end of round $\diam(G)+1$.

For the lower bound, we provide a cooling sequence that cools $G$ in at least $\Big\lceil\frac{\diam(G)+2}{2}\Big\rceil$ rounds.
Let $(v_0, v_1,\ldots, v_d)$ be a path of diameter length in $G$. The cooling sequence will be $\left(v_{2i-2} : 1 \leq i \leq \Big\lceil\frac{\diam(G)+2}{2}\Big\rceil\right)$.

Assume that at the start of round $i\geq 2$, each cooled node has distance at most $2i-4$ from $v_0$.
After the cooling spreads in this round, every cooled node has distance at most $2i-3$ from $v_0$. As such, the node $v_{2i-2}$ is uncooled and a possible choice for the next node in the cooling sequence.
This node is now cooled, so starting round $i+1$, each cooled node has distance at most $2(i+1)-4$ from $v_0$.
This sets up the recursion, where we note that the condition holds at the start of the second round. Since $2\Big\lceil\frac{\diam(G)+2}{2}\Big\rceil-2 \leq \diam(G)+1$, the sequence provided is a cooling sequence.
\end{proof}

We now apply these bounds to give that for the path $P_n$ of order $n,$  \[\cool{P_n}=\Big\lceil\frac{n+1}{2}\Big\rceil = \Big\lceil\frac{\diam(P_n)+2}{2}\Big\rceil.\]
The upper bound follows from Theorem~\ref{thm:upper_n_div_2} and the lower bound from Theorem~\ref{thm:upper_diam_cooling}.
In particular, the upper bound of Theorem~\ref{thm:upper_n_div_2} is tight.
In passing, we note (with proof omitted) that for a cycle $C_n$, $\cool{C_n} = \ceil{\frac{n+2}{3}}$.

Using a more complicated example, we can show that the upper bound of Theorem~\ref{thm:upper_diam_cooling} is also tight. Define the \emph{complete caterpillar} of length $d$, \(\mathrm{CC}_d\), as the graph formed by appending one node to each non-leaf node of \(P_d\). We call the nodes of \(P_d\) in \(\mathrm{CC}_d\) the \emph{spine} of the caterpillar. Note that $\mathrm{CC}_d$ has $n=2d-2$ nodes.

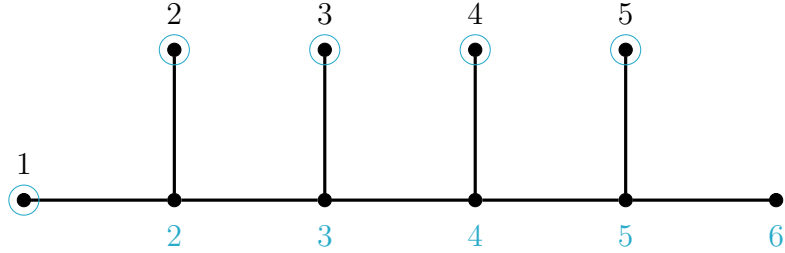
\begin{figure}[h]
\centering
\begin{tikzpicture}
\begin{scope}[every node/.style={circle,draw,fill=black,inner sep=0pt,minimum size=5pt}]
    \node (A) at (0,0) {};
    \node (B) at (2,0) {};
    \node (C) at (4,0) {};
    \node (D) at (6,0) {};
    \node (E) at (8,0) {};
    \node (F) at (10,0) {};
    \node (G) at (2,2) {};
    \node (H) at (4,2) {};
    \node (I) at (6,2) {};
    \node (J) at (8,2) {};
\end{scope}

\begin{scope}[>={Stealth[black]},
              every node/.style={fill=black,circle},
              every edge/.style={draw=black,very thick}]
    \path [-] (A) edge (B);
    \path [-] (B) edge (C);
    \path [-] (C) edge (D);
    \path [-] (D) edge (E);
    \path [-] (E) edge (F);
    \path [-] (B) edge (G);
    \path [-] (C) edge (H);
    \path [-] (D) edge (I);
    \path [-] (E) edge (J);
    \path[draw=BLUE_D] (A) circle[radius=0.2];
    \path[draw=BLUE_D] (G) circle[radius=0.2];
    \path[draw=BLUE_D] (H) circle[radius=0.2];
    \path[draw=BLUE_D] (I) circle[radius=0.2];
    \path[draw=BLUE_D] (J) circle[radius=0.2];
\end{scope}
\begin{scope}
    \coordinate (A) at (0,0);
    \node at (A) [above = 2mm of A] {$1$};
    \coordinate (B) at (2,0);
    \node at (B) [below = 2mm of B,BLUE_D] {$2$};
    \coordinate (C) at (4,0);
    \node at (C) [below = 2mm of C,BLUE_D] {$3$};
    \coordinate (D) at (6,0);
    \node at (D) [below = 2mm of D,BLUE_D] {$4$};
    \coordinate (E) at (8,0);
    \node at (E) [below = 2mm of E,BLUE_D] {$5$};
    \coordinate (F) at (10,0);
    \node at (F) [below = 2mm of F,BLUE_D] {$6$};
    \coordinate (G) at (2,2);
    \node at (G) [above = 2mm of G] {$2$};
    \coordinate (H) at (4,2);
    \node at (H) [above = 2mm of H] {$3$};
    \coordinate (I) at (6,2);
    \node at (I) [above = 2mm of I] {$4$};
    \coordinate (J) at (8,2);
    \node at (J) [above = 2mm of J] {$5$};
\end{scope}
\end{tikzpicture}
\caption{An example of cooling on the complete caterpillar of length 6. Black labels indicate the nodes of the cooling sequence in increasing order. Blue labels indicate the round that the corresponding node was cooled.}
    \label{fig:CC6_cooling_example}
\end{figure}

\begin{theorem} \label{lem:caterpillar_cooling_number}
We have that     \[\cool{\mathrm{CC}_d} = \Big\lceil\frac{n+1}{2}\Big\rceil = \diam(P_n)+1.\]
\end{theorem}
\begin{proof}
The upper bound follows from Theorem~\ref{thm:upper_n_div_2}, and so to find the lower bound, we provide a strategy that takes $\lceil\frac{n+1}{2}\rceil$ rounds. Let $(v_1, v_2,\ldots, v_{d})$ be the spine of the caterpillar, and let $v_i'$ be the additional node that is appended to $v_i$, for $2 \leq i \leq d-1$. The cooling sequence will be $v_1$ followed by $(v_{i}' : 2 \leq i \leq d-1)$.

On round $1$, the first node in the cooling sequence $v_1$ is cooled.
At the start of round $i\geq 2$, the nodes $\{v_1, \ldots, v_{i-1}\} \cup \{v_2', \ldots, v_{i-1}'\}$ have been cooled.
The cooling spreads, which cools only the node $v_i$.
The next node in the cooling sequence is $v_i'$.
At the start of round $i+1$, the nodes $\{v_1, \ldots, v_{i}\} \cup \{v_2', \ldots, v_{i}'\}$ have been cooled. This argument recursively repeats and ends on round $d$ when $v_d$ is cooled as cooling spreads.
We then have that \(\cool{\mathrm{CC}_d} = d =  \lceil\frac{n+1}{2}\rceil\), as required.
\end{proof}

We finish the section by noting that determining the cooling number of certain graph families appears challenging. Even for spiders in general, determining the exact cooling number is not obvious. The following result provides bounds and exact values of the cooling numbers of certain spiders. We refer to the paths attached to the root node of a spider as its \emph{legs}.

\begin{theorem} \label{thm:spider_log}
Let $T$ be a spider with $2m$ legs, each of length $r$.  If we have that $m < \lceil \log_2{r+1}\rceil$, then
    \[\cool{T}\geq 2 \sum_{1 \leq i \leq m} \Big\lfloor \frac{r+1}{2^i} \Big\rfloor \sim (1-1/2^m)2r.\]
    Otherwise,
     \[\cool{T} = \diam(S) +1.\]
\end{theorem}

\begin{proof}
We partition the cooling process into two major phases: before the head of the spider is cooled and after it is cooled.
Let \[m' = \min\{m,\lceil \log_2{r+1}\rceil\}.\]  We further split each phase into $m'$ subphases.

We start in the first phase. Suppose we enter subphase $i$, for $i$ from $1$ up to $m'$. This subphase runs for $\Big\lfloor \frac{r+1}{2^{(m'+1-i)}}\Big\rfloor$ rounds. In each round, we add the uncooled node in leg $i$ furthest from the head to the cooling sequence. Note that after subphase $i$, each leg $i'<i$ will have at most $\frac{r+1}{2^{(m'-i)}}\leq r$ cooled nodes.
At the end of this phase, there are $m'$ legs with at most $r$ cooled nodes, and the remaining at least $m'$ legs do not have any cooled nodes.
Note that the head of the spider has not yet been cooled.

The next phase begins. Suppose we enter subphase $i$ of the second phase, for $i$ from $1$ up to $m'$.  This subphase runs for $\lfloor \frac{r+1}{2^{i}}\rfloor$ rounds.
In each round, we add the uncooled node in leg $i$ to the cooling sequence that is closest to the head. Note that after subphase $i$, each leg $i' \leq i$ can be assumed to be completely cooled, while legs $i'>i$ will have $\sum_{1 \leq j \leq i} \big(\lfloor \frac{r+1}{2^{j}} \rfloor \big)-1  < r$ cooled nodes (recalling that the cooling spreads at the start of a round).

A total of $2\sum_{1 \leq i \leq m'} \big(\lfloor \frac{r+1}{2^{i}} \rfloor \big)$ rounds have been played, completing the result when $m< \lceil \log_2{r+1}\rceil$. If $m \geq \lceil \log_2{r+1}\rceil$, then note that the head was cooled in round $r+1$ and that leg $m$ did not contain a node in the cooling sequence. Therefore, the leaf of leg $m$ was cooled on round $2r+1 = \diam(S)+1.$   \end{proof}

\section{Isoperimetric results and grids}

This section studies cooling on Cartesian grids using isoperimetric results. Burning on Cartesian grids remains a difficult problem, with only bounds available in many cases. See \cite{fence,PP} for results on burning Cartesian grids.

Suppose $G$ is a graph. For a $S \subseteq V(G)$, define its \emph{node border}, $\border(S)$, to be the set of nodes in $V(G) \setminus S$ that neighbor nodes in $S$.  We then have the \emph{node-isoperimetric parameter} of $G$ at $s$ as $\Phi_V(G,s) = \min_{S:|S|=s} |\border(S)|$, and the \emph{isoperimetric peak} of $G$ as $\Phi_V(G) = \max_s \{\Phi_V(G,s)\}$. Note that in other work, it is common to use $\delta(S)$ in place of $\border(S)$, but we use the chosen notation to prevent confusion with the minimum degree of a graph, $\delta(G)$.

Literature around node borders often either focuses on Cheeger's inequality, which is based on the ratio of $|\border(S)|$ to $|S|$, or focuses on isoperimetric inequalities, which are bounds on $\Phi_V(G,s)$ from below.
In recent work, for positive $x$ and $y$, an inequality that bounds the maximum difference between the node-isoperimetric parameter at $x$ and at $x+y$ was given while proving a distinct result (see Theorem~5 of \cite{kvis}). We provide an analogous, reworded version of this bound with full proof for completeness.

\begin{lemma}[\cite{kvis}] \label{lem:isoperi_smooth_change}
For a graph $G$ and integers $x,y\geq 0$,
\[
\Phi_V(G,x)-y \leq \Phi_V(G,x+y).
\]
\end{lemma}
\begin{proof}
Let $S_{x+y}$ be a set of nodes of cardinality $x+y$ with $|\border (S_{x+y})| = \Phi_V(G,x+y)$.
Note that if a node $u$ is removed from $S_{x+y}$, then the only node that may be in the border of the new set $S_{x+y} \setminus \{u\}$ that was not in the border of $S_{x+y}$ is the node $u$ itself, as the nodes in the border are neighbors of nodes in the set.
It follows that $\border(S_{x+y}\setminus \{u\}) \subseteq \border(S_{x+y}) \cup \{u\}$.

By a similar argument, if we remove any set $S_y \subseteq S_{x+y}$ of cardinality $y$ from $S_{x+y}$, we have
$\border(S_{x+y} \setminus S_y) \subseteq \border(S_{x+y}) \cup S_y$.
But this gives $| \border(S_{x+y} \setminus S_y)| \leq |\border(S_{x+y})|+y$, which re-arranges to yield
$$| \border(S_{x+y} \setminus S_y)| - y \leq |\border(S_{x+y})| = \Phi_V(G,x+y).$$
This completes the proof since $\Phi_V(G,x) \leq |\border(S_{x+y} \setminus S_y)|$.
\end{proof}

Lemma~\ref{lem:isoperi_smooth_change} can be considered a relative isoperimetric inequality.
When we define the set of cooled nodes at time $i$ as $S_i$ during cooling, this relative isoperimetric inequality yields a sequence of values, $x_i$, such that $|S_i|\geq x_i$ for all $i$, independent of the strategy that was used to cool the graph.
This sequence of values derives a natural upper bound on $\cool{G}$ as follows. We define $x_1=1$, and let $$x_{i+1} = x_i + \Phi_V(G,x_i)+1.$$
Suppose $I$ is the smallest value with $x_I \geq |V(G)|$.

\begin{theorem} \label{thm:main_isoperimetry}
If $G$ is a graph, then we have that
\[
\cool{G} \leq I.
\]
\end{theorem}
\begin{proof}
Suppose for the sake of contradiction that $\cool{G} \geq I+1$. For some optimal cooling strategy, let $S_j$ be the set of cooled nodes at the end of round $j$.  We then have that $|S_1|=1$.
Assume that there is some round $J$ where $|S_{J-1}|\geq x_{J-1}$ but $|S_{J}|<x_{J}$.
Define $y = |S_{J-1}|-x_{J-1} \geq 0$.
It follows by the definition of the cooling process, the definition of $\Phi_V$, and by Lemma~\ref{lem:isoperi_smooth_change} that
\begin{eqnarray*}
|S_J| & = &|S_{J-1}| + |\border(S_{J-1})| +1 \\
& \geq & x_{J-1} + y + \Phi_V(G,x_{J-1}+y) + 1 \\
& \geq & x_{J-1} + \Phi_V(G,x_{J-1}) -y + y+1 \\
& = & x_{J-1}+\Phi_V(G,x_{J-1})+1 = x_{J}.
\end{eqnarray*}
This contradicts the fact that $|S_J| < x_J$, and so we are done.
\end{proof}

The (Cartesian) grid graph of length $n$, written $G_n,$ is the graph with nodes of the form $(u_1,u_2)$, where $1\le u_1,u_2 \le n$, and an edge between $(u_1,u_2)$ and $(v_1,v_2)$ if $u_1=v_1$ and $|u_2-v_2|=1$, or if $u_2=v_2$ and $|u_1-v_1|=1$. The following total ordering of the nodes in the grid, called the \emph{simplicial ordering}, can be found in \cite{BI}.
For two nodes $u=(u_1, u_2)$ and $v=(v_1, v_2)$, define $u<v$ when either $u_1+u_2 < v_1 + v_2$, or $u_1 + u_2 = v_1 + v_2$ and $u_1<v_1$. Let $S_i$ be the $i$ smallest nodes under this total ordering.

\begin{lemma}[\cite{BI}] \label{lem:grid_isometric} No $i$-subset of $V(G)$ has a smaller node border than $S_i$, and so $|\border(S_i)| = \Phi_V(G,i)$.
\end{lemma}

Under the simplicial ordering, adding the smallest node in $V(G)\setminus S_i$ to $S_i$ yields the set $S_{i+1}$.
The set of nodes $\border(S_i)$ is the set containing the $|\Phi_V(G,i)|$
smallest nodes that are larger than the nodes in $S_i$.
As such, $S_i \cup \border(S_i) = S_{i+\Phi_V(G,i)}$ and it also follows that $|S_i \cup \border(S_i)| = i+\Phi_V(G,i)$.

\begin{theorem} \label{thm:isoperimetric_grid_strategy}
An optimal cooling strategy for a grid graph $G_n$ is formed by choosing the next node in a cooling sequence to be the smallest node in the simplicial ordering that has not yet been cooled.
\end{theorem}
\begin{proof}
Let $U_i$ be the set of nodes cooled by the end of round $i$ when this strategy was performed.
Note that for each $i$, there exists an $x_i$ with $S_{x_i} = U_i$.
We then have that that $|U_i|=x_i$.
During the next round, the cooling will spread to $\border(U_i)$, meaning the nodes in $U_i \cup \border(U_i) = S_{x_i+\Phi_V(G,x_i)}$ are now cooled.

We then cool the smallest uncooled node, implying that exactly the nodes $S_{x_i+\Phi_V(G,x_i)+1}$ are cooled, and so $U_{i+1} = S_{x_i+\Phi_V(G,x_i)+1}$.
It follows that $$x_{i+1} = |U_{i+1}| = x_i+\Phi_V(G,x_i)+1.$$
If $I$ is the smallest value with $x_I \geq |V(G)|$, then on all rounds $i<I$, $|U_i|<|V(G)|$, and so there is still some uncooled node at the end of round $i$, and so play continues into round $I$.
The process will, therefore, terminate under this approach exactly on round $I$.
This proves that $\cool{G} \geq I$.
Theorem~\ref{thm:main_isoperimetry} finishes the proof.
\end{proof}

We explicitly determine the cooling number for Cartesian grids up to a small additive constant in the following result.

\begin{figure}
    \centering
\includegraphics[scale=0.2]{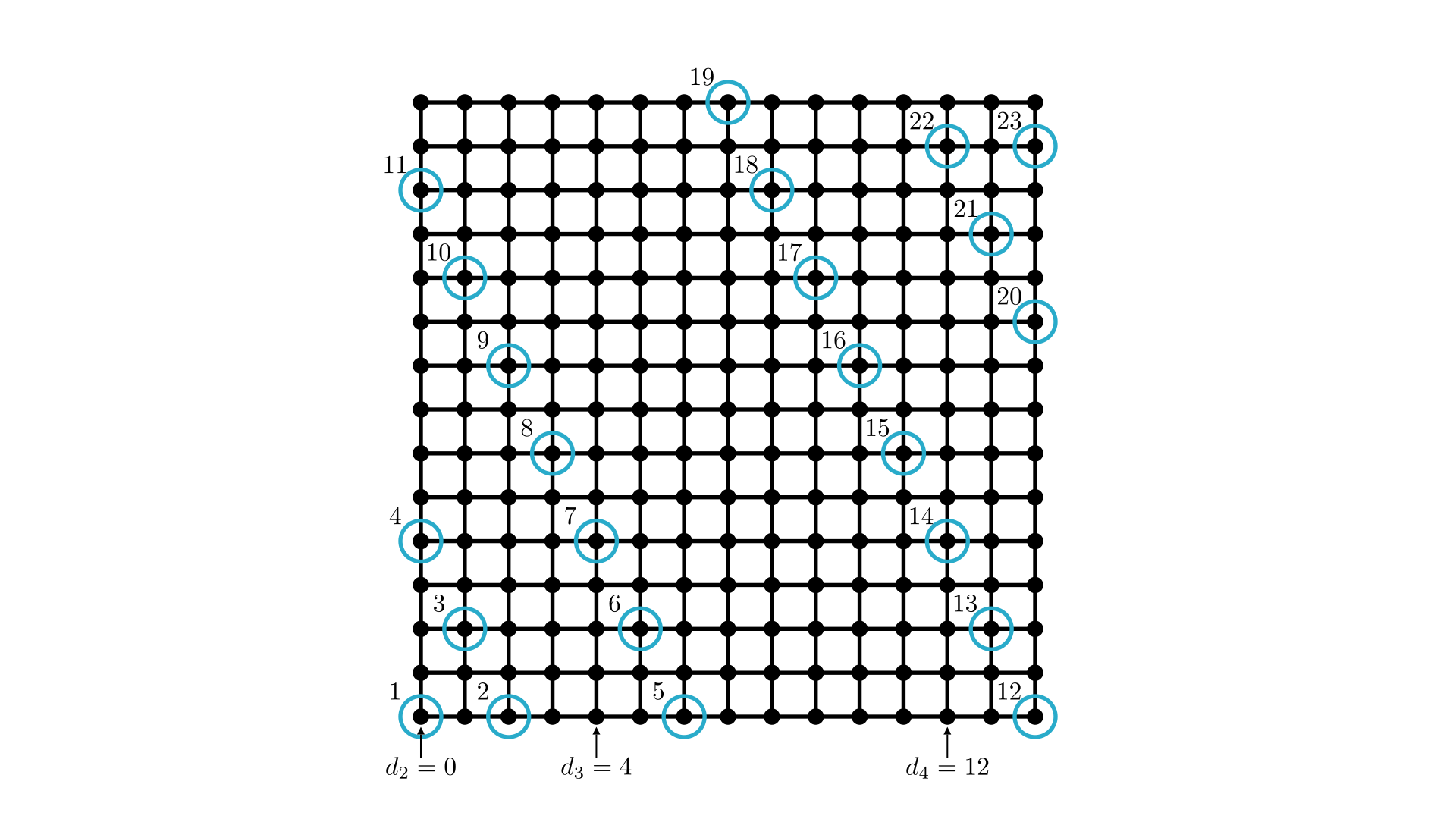}
    \caption{An example of the cooling sequence for the $15 \times 15$ grid using the strategy of Theorem~\ref{thm:cooling_grids_main}.
    The nodes on the first row with distance $d_2,d_3,d_4$ from $(1,1)$ are also indicated.
    }
    \label{fig:cooling_grids_main}
\end{figure}

\begin{theorem} \label{thm:cooling_grids_main}
For each $n\geq 1$, there is an $\varepsilon \in \{0,1,2\}$ so that
    \[
    \cool{G_n} = 2n-2\lfloor \log_2(n+3) \rfloor + \varepsilon.
    \]
\end{theorem}
\begin{proof}
We analyze the strategy of Theorem~\ref{thm:isoperimetric_grid_strategy} in phases, which we know is an optimal cooling strategy.
Phase $1$ is just the first round, where only the node $(1,1)$ is cooled.
Phase $p$, with $2 \leq p \leq \lfloor \log_2(n+3) \rfloor$ starts with exactly those nodes of distance at most $d_p=2^{p}-4$ from $(1,1)$ being cooled, and this phase will last for exactly $d_p+3$ rounds.

In round $t$ of phase $p$, with $1 \leq t\leq d_p+3$, the cooling spreads to all uncooled nodes of distance $d_p+t$ from $(1,1)$ and to the nodes $\{(t',d_p+t+3-t'): 1 \leq t' \leq 2t-2\}$, which each have distance $d_p+t+1$ from $(1,1)$.
The node $(2t-1,d_p+4-t)$ is then chosen to be the next node of the cooling sequence, so the set of cooled nodes is exactly those nodes with distance at most $d_p+t$ from $(1,1)$ and the nodes $\{(t',d_p+t+3-t'): 1 \leq t' \leq 2t-1\}$, which have distance $d_p+t+1$ from $(1,1)$.
At the end of round $d_p+3$, this is exactly the set of nodes of distance $2d_p+4 = 2^{p+1}-4 = d_{p+1}$, and so we may iterate this procedure until phase $\lfloor \log_2(n+3) \rfloor$.

Note that between the start and end of phase $p$, the ball of cooled nodes about $(1,1)$ grows in radius by $2^p$ even though only $2^p-1$ rounds have occurred.
Thus, if $T$ rounds have occurred from the start of play until some round during phase $p$, then the cooled nodes all sit within a ball of radius $T+p-2$ around $(1,1)$, and all the nodes in this ball will be cooled if the round is the last of phase $p$.

At the end of phase $p= \lfloor \log_2(n+3) \rfloor-1$, note that the cooled nodes form a ball of radius $r$ around $(1,1)$, where $\frac{n}{2}-2 \leq r \leq n-1$, and after phase $p$, the cooled nodes form a ball of radius $r$ around $(1,1)$, where $r \geq n$.

We will briefly discuss a different strategy that will help us bound the time the optimal cooling process will take.
Suppose we have played the above strategy for the first $\lfloor \log_2(n+3) \rfloor-1$ phases, and for phase $\lfloor \log_2(n+3) \rfloor$, we terminate the phase after the first round where a node of distance $n-2$ from $(1,1)$ was cooled.

Let $T$ denote the number of rounds that have occurred by the end of this last modified phase.
Note that $T+\lfloor \log_2(n+3) \rfloor -2 = n-2$
We construct an additional $T$ rounds for our modified strategy, which naturally split into $\lfloor \log_2(n+3) \rfloor$ phases.
If node $(r_i,c_i)$ was played on round $i$ with $1 \leq i \leq T$, then we play node $(n+1-r_i,n+1-c_i)$ on round $2T+1-i$, for $1 \leq i \leq T$.

If phase $p$ consisted of rounds $t_1$ through to $t_2$, then phase $2(\lfloor \log_2(n+3) \rfloor-1)+1-p$ consists of rounds $2T+1-t_2$ through to $2T+1-t_1$.
A similar analysis yields that each such node in the cooling sequence is uncooled before we cool it; hence, this is a cooling sequence, and cooling on $G_n$ using this modified approach lasts for at least $2T$ rounds.
If we proceed optimally, then the graph would take at least $2T$ rounds to cool.
Noting that $T+\lfloor \log_2(n+3) \rfloor -2 = n-2$, we have that $2T = 2n - 2\lfloor \log_2(n+3) \rfloor$.

Suppose, for the sake of contradiction, that the optimal approach lasts for at least $2T+3$ rounds.
Note that after $T+3$ rounds, some nodes of distance $n+1$ from $(1,1)$ must have been cooled, and all nodes of distance $n$ from $(1,1)$ must have been cooled

We now consider the reverse strategy, which first plays the last choice made in the optimal strategy, and so on.
Note that by symmetry, after the first $T$ rounds of this reverse strategy, the node that was chosen to be cooled has distance at least $n-2$ from $(n,n)$.
However, this means that the node chosen to be cooled on round $T$ of this reverse strategy has distance at most $n$ from $(1,1)$.
But going back to the original optimal strategy, this means that round $T+4$ must have been played at a distance of at most $n$ from $(1,1)$, but all of these nodes were cooled by round $T+3$.
This gives us the desired contradiction, so the optimal strategy must last for at most $2T+2$ rounds.  Therefore, the optimal strategy lasts for $ 2T $, $ 2T+1 $, or $ 2T+2 $ rounds, and the proof is complete.
\end{proof}

\section{Cooling the ILT model}

Motivated by structural balance theory, the \emph{Iterated Local Transitivity} (or \emph{ILT} model) iteratively adds transitive triangles over time; see \cite{ilt}.
Graphs generated by these models exhibit several properties observed in complex networks, such as densification, small-world properties, and bad spectral expansion.
The ILT model takes a graph $G=G_0$ as input and defines $\mathrm{ILT}(G)$ by adding a cloned node $x'$ for each node $x$ in $G$, and making $x'$ adjacent to the original node $x$ and the neighbors of $x$.
The set of cloned nodes is referred to as $\mathrm{CLONE}(G)$, and we write $\mathrm{ILT}_t(G)$ to denote the graph obtained by applying the ILT process $t$ times to $G$.
When $t=1$, we typically write $\mathrm{ILT}(G)$.
We call any graph produced using at least one iteration of the ILT process an \emph{ILT graph}.
The diameter of $\mathrm{ILT}(G)$ is the same as that of $G$ unless the diameter of $G$ is 1, in which case the diameter of $\mathrm{ILT}(G)$ is $2$.

The burning numbers of an ILT graph either equals $b(G)$ or $b(G)+1$; see \cite{BRJ2}.  Even though the order of the graphs generated by the ILT model grows
exponentially with time, the burning number of ILT graphs remains constant. We now consider the cooling number of ILT graphs.

\begin{theorem}\label{thm:ILT_only_increases}
If $G$ is a graph, then $\mathrm{CL}(\mathrm{ILT}(G)) \geq \mathrm{CL}(G)$.
\end{theorem}

\begin{proof}
 Let $(v_1,v_2,\ldots,v_k)$ be a maximum-length cooling sequence for $G$.
Since the distance between $v_i$ and $v_j$ is the same in both $G$ and in $\mathrm{ILT}(G)$, the sequence $(v_1,v_2,\ldots,v_k)$ is also a cooling sequence for $\mathrm{ILT}(G)$.
If $\mathrm{CL}(G)=k$, then as the cooling process on $\mathrm{ILT}(G)$ with the given cooling sequence lasted $k$ rounds, we have the result.

Otherwise, we may assume that $\mathrm{CL}(G)=k+1$ and that some node, say $v$, was one of the nodes cooled during the last round when cooling $G$.
Since $v$ must have a distance at least $k+1-i$ from $v_i$ in $G$ for this to occur, and since the distance from $v$ to $v_i$ is the same as this in $\mathrm{ILT}(G)$, it follows that $v$ must have a distance at least $k+1-i$ from $v_i$ in $\mathrm{ILT}(G)$.
However, this implies that at the end of round $k$ of the cooling process on $\mathrm{ILT}(G)$, node $v$ has not yet been cooled.
Thus, the cooling process on $\mathrm{ILT}(G)$ lasts at least $k+1 =\mathrm{CL}(G)$ rounds, and so the proof is complete.

\end{proof}

We next investigate the cooling number of the graphs $\mathrm{ILT}_t(P_n)$, where $t\geq 1$ and $P_n = (p_1,\ldots,p_n)$ is the path graph on $n$ nodes. See Figure \ref{fig:ILTP6-cooling-example} for an illustration of $\mathrm{ILT}_1(P_6)$. For $1 \leq k \leq t$, the $k$-th \textit{layer set} of $p_i$ is the set $L_k(p_i) = L_{k-1}(p_i) \cup \text{CLONE}(L_{k-1}(p_i))$, where $L_0(p_i) = \{p_i\}$. When $k=t$, we write $L(p_i)$ instead.

\begin{figure}[t]
\centering
\begin{tikzpicture}
\begin{scope}[every node/.style={circle,draw,fill=black,inner sep=0pt,minimum size=5pt}]
    \node (1) at (0,0) {};
    \node (2) at (2,0) {};
    \node (3) at (4,0) {};
    \node (4) at (6,0) {};
    \node (5) at (8,0) {};
    \node (6) at (10,0) {};
    \node (12) at (0,2) {};
    \node (22) at (2,2) {};
    \node (32) at (4,2) {};
    \node (42) at (6,2) {};
    \node (52) at (8,2) {};
    \node (62) at (10,2) {};
\end{scope}

\begin{scope}[>={Stealth[black]},
              every node/.style={fill=black,circle},
              every edge/.style={draw=black,very thick}]
    \path [-] (1) edge (2);
    \path [-] (2) edge (3);
    \path [-] (3) edge (4);
    \path [-] (4) edge (5);
    \path [-] (5) edge (6);
    \path [-] (1) edge (12);
    \path [-] (1) edge (22);
    \path [-] (2) edge (22);
    \path [-] (2) edge (32);
    \path [-] (3) edge (32);
    \path [-] (3) edge (42);
    \path [-] (4) edge (42);
    \path [-] (4) edge (52);
    \path [-] (5) edge (52);
    \path [-] (5) edge (62);
    \path [-] (6) edge (62);
    \path [-] (2) edge (12);
    \path [-] (3) edge (22);
    \path [-] (4) edge (32);
    \path [-] (5) edge (42);
    \path [-] (6) edge (52);
    \path[draw=BLUE_D] (12) circle[radius=0.2];
    \path[draw=BLUE_D] (22) circle[radius=0.2];
    \path[draw=BLUE_D] (42) circle[radius=0.2];
    \path[draw=BLUE_D] (52) circle[radius=0.2];
\end{scope}

\begin{scope}
    \coordinate (1) at (0,0);
    \node at (1) [below = 2mm of 1,BLUE_D] {$2$};
    \coordinate (2) at (2,0);
    \node at (2) [below = 2mm of 2,BLUE_D] {$2$};
    \coordinate (3) at (4,0);
    \node at (3) [below = 2mm of 3,BLUE_D] {$3$};
    \coordinate (4) at (6,0);
    \node at (4) [below = 2mm of 4,BLUE_D] {$4$};
    \coordinate (5) at (8,0);
    \node at (5) [below = 2mm of 5,BLUE_D] {$4$};
    \coordinate (6) at (10,0);
    \node at (6) [below = 2mm of 6,BLUE_D] {$5$};
    \coordinate (12) at (0,2);
    \node at (12) [above = 2mm of 12] {$1$};
    \coordinate (22) at (2,2);
    \node at (22) [above = 2mm of 22] {$2$};
    \coordinate (32) at (4,2);
    \node at (32) [above = 2mm of 32,BLUE_D] {$3$};
    \coordinate (42) at (6,2);
    \node at (42) [above = 2mm of 42] {$3$};
    \coordinate (52) at (8,2);
    \node at (52) [above = 2mm of 52] {$4$};
    \coordinate (62) at (10,2);
    \node at (62) [above = 2mm of 62,BLUE_D] {$5$};
\end{scope}
\end{tikzpicture}
\caption{An example of cooling on $\mathrm{ILT}(P_6)$. Black
labels indicate the nodes of the cooling sequence in increasing order. Blue
labels indicate the round that the corresponding node was cooled.}
\label{fig:ILTP6-cooling-example}
\end{figure}

\begin{theorem}
If $n \geq 3$ and $t \ge 1$, then
    $$\mathrm{CL}(\mathrm{ILT}_t(P_n)) =
    \begin{cases}
        \left\lceil 2n/3 \right\rceil   & \text{ if }t = 1 \normalfont{\text{ and }} n \equiv 2 \pmod{n};\\
        \left\lceil 2n/3 \right\rceil  +1& \normalfont{\text{ otherwise}.}
    \end{cases}
    $$
\end{theorem}
\begin{proof}
Let $G=\mathrm{ILT}_t(P_{n})$ and let $(v_1,\ldots,v_k)$ be a cooling sequence for $G$.
We begin with an upper bound on $k$.
As the layer sets $\{L(p_i)\}$ form a partition of $V(G)$, each element $v_i$ of the cooling sequence is in exactly one $L(p_j)$, for some $1 \leq j \leq n$.
We claim that any three consecutive layer sets, those of the form $L(p_j), L(p_{j+1}), L(p_{j+2})$, can contain at most two sequence elements total.

Indeed, suppose that $v,w$ are nodes belonging to any of these three layer sets.  Since $p_j$ and $p_{j+2}$ are adjacent to $p_{j+1}$ in $P_{n}$ and by definition of the ILT process, $v$ and $w$ are adjacent to $p_{j+1}$ in $G$.
As such, $d(v,w) \leq 2$.
If a node $u$ in any of those layer sets were cooled, all three layer sets would be fully cooled in two rounds.
In each round, the selection of a node in the cooling sequence happens after the cooling has spread from the nodes that were cool by the end of the previous round, so only one node in the three layer sets could be selected after $u$ was cooled, proving the claim.
Thus, as there are $n$ layer sets, we can separate them into $\lfloor n/3\rfloor$ consecutive triples, each of which have at most two cooled members.
If $n$ is not divisible by 3, then the remaining layer sets may also each have a cooled element.
Now we have that $k \leq \left\lceil 2n/3 \right\rceil$.

This upper bound is also achieved, which we show by constructing a cooling sequence $(w_1,\ldots,w_k)$ of the desired length.
Let the clone of the node $p_i \in V(\mathrm{ILT}_{t-1}(P_n))$ that was created in the $t$-th application of the ILT process be denoted $p_i'$.
For $i \geq 1$, define $w_i = p_{i+\lfloor (i-1)/2 \rfloor}'$.
This sequence has length $\left\lceil 2n/3 \right\rceil$, and we claim it is a cooling sequence.
Indeed, for any $1 \leq i< j \leq n$ we have that $d(p_i',p_j') = \max(2,|i-j|)$, as the path $(p_i',p_{i+1},p_{i+2},\ldots,p_{j-1},p_j')$ is of this length and any path between nodes in layer sets $L(p_i)$ and $L(p_j)$ must cross through every layer set $L(p_k)$ with $i < k < j$, implying no shorter path can exist.
Thus, the sequence is a cooling sequence of length $\left\lceil 2n/3 \right\rceil$, and as this matches the upper bound, this cooling sequence is optimal.

We have shown that the length of an optimal cooling sequence of $G$ is $\left\lceil 2n/3 \right\rceil$. However, it is possible that the cooling process could last one round longer: this breaks into two cases, with whose study we conclude the proof.

\vspace{0.1in}

\noindent \emph{Case 1:} $t = 1$  and $n \equiv 2 \pmod{3}$. Consider the cooling sequence $(w_1,w_2,\ldots,w_k)$ as defined above.
As $n \equiv 2 \pmod{3}$, the final cooled node $w_k$ is the node $p_n'$, and the node $p_n$ is cooled at the beginning of round $k$. Thus, there is no additional round of cooling, showing that $\mathrm{CL}(G) = \left\lceil 2n/3 \right\rceil.$

\vspace{0.1in}

\noindent \emph{Case 2:} Otherwise. Consider the cooling sequence $(w_1,w_2,\ldots,w_k)$ as defined above.
Here, we have that either $w_k\neq p_n'$ or that $\mathrm{CLONE}(L_{t-1}(p_n))\setminus \{p_n'\}$ is nonempty. This implies that after the cooling of $w_k$, there is at least one uncooled element of $L(p_n)$, and an extra round of cooling is needed.
 \end{proof}

The next result shows that the second time-step determines the cooling number of ILT graphs.

\begin{theorem}\label{thrm:iltcooling}
For any graph $G$, the maximum length of a cooling sequence in $\mathrm{ILT}_2(G)$ and $\mathrm{ILT}_t(G)$ are the same.
\end{theorem}
\begin{proof}
Suppose that $(u_1, u_2,\ldots, u_d)$ is an optimal cooling sequence in $\mathrm{ILT}_t(G)$.
Given a node $u$ in $\mathrm{ILT}_t(G)$, let $f(u)$ be the node in $G$ that was cloned to form $u$.
Since $t \geq 2$, for each $x \in V(G)$, there are at least two distinct clone nodes in $\mathrm{ILT}_t(G)$ that were created in the latest iteration of the $\mathrm{ILT}$ process, say $x'$ and $x''$, such that $f(x')=f(x'')=x$.

We define a sequence of nodes in $\mathrm{ILT}_2(G)$ of length $d$, $(v_1, \ldots, v_d)$, as follows.
If $f(u_i) \neq f(u_{i-1})$, then define $v_i = f(u_i)'$; otherwise, define $v_i = f(u_i)''$.
We will show that this is an optimal cooling sequence for $\mathrm{ILT}_2(G)$.

To begin, we show that this is a valid cooling sequence.
Assume for a contradiction that this sequence is not a valid cooling sequence. As such, there must be some $v_j$ that will already be cool when we try to cool it on round $j$. For this to have occurred, there must some $i<j$ and some path from $v_i$ to $v_j$ of length at most $j-i$, say $(p_i, p_{i+1}, \ldots, p_j)$, such that node $p_r$ was cooled on round $r$ for each $r\in \{i, i+1, \ldots, j\}$.
The walk $(f(p_i), f(p_{i+1}), \ldots, f(p_j))$ in $\mathrm{ILT}_2(G)$ has length $j-i$, and so contains a path of length $r' \leq j-i$.
For simplicity, we may assume that $(f(p_i), f(p_{i+1}), \ldots, f(p_j))$ is a path in $\mathrm{ILT}_2(G)$.
Since $f(p_i) = f(v_i) = f(u_i)$ and $f(p_j) = f(v_j)=f(u_j)$, it follows that $(u_i, f(p_{i+1}), \ldots, f(p_{j-1}), u_j)$ is a path of length $j-i$ in $\mathrm{ILT}_t(G)$. Since $u_i$ was cooled in round $i$ while cooling the graph $\mathrm{ILT}_t(G)$, and the cooling spreads to each uncooled neighbor over each round, it follows that $u_j$ must have been cooled within $j-i$ rounds. Then $u_j$ was cooled by round $i+(j-i) = j$, so $u_j$ was already cooled by round $j$.  However, $(u_1, u_2,\ldots, u_d)$ is a cooling sequence in $\mathrm{ILT}_t(G)$, and so we have the required contradiction.

To show that the cooling sequence $(v_1, \ldots, v_d)$ in $\mathrm{ILT}_2(G)$ is optimal, assume that there is some cooling sequence $(v_1, \ldots, v_{d+1})$ in $\mathrm{ILT}_2(G)$. This sequence of nodes is also a sequence of nodes in $\mathrm{ILT}_t(G)$, and further, the distance between any of these nodes is the same in both $\mathrm{ILT}_2(G)$ and $\mathrm{ILT}_t(G)$.
In $\mathrm{ILT}_t(G)$, since no cooling sequence can have length $d+1$, there must be a pair of nodes $v_i$ and $v_j$, $i<j$, with $d(v_i,v_j)\leq j-i$.
 (This follows by a similar argument to the first part of the proof.)
We then have that $v_j$ will already be cooled by the time we try to cool it on turn $j$, contradicting that this is a cooling sequence.
Thus, the maximum length of a cooling sequence in $\mathrm{ILT}_2(G)$ is $d$, and the proof follows.
\end{proof}

If a graph $G$ has a maximum cooling sequence of length $s$, then $\cool{G}\in \{s,s+1\}$. Theorem~\ref{thm:ILT_only_increases} assures us that applying the ILT process only increases the cooling value. We thus have the following result proving that, as in the case of burning, the cooling number remains bounded by a constant throughout the ILT process.

\begin{corollary}\label{cor:ILT}
For any graph $G$, we have that $$\cool{\mathrm{ILT}_2(G)} \le \cool{\mathrm{ILT}_t(G)} \le \cool{\mathrm{ILT}_2(G)}+1.$$
\end{corollary}

In particular, Theorem~\ref{thm:ILT_only_increases} and Corollary~\ref{cor:ILT} together imply that for every graph $G$, either $\cool{\mathrm{ILT}_t(G)} = \cool{\mathrm{ILT}_2(G)}$ for all $t\geq 2$, or there exists a threshold value $t_0$ such that the cooling number of $\mathrm{ILT}_t(G)$ is equal to $\cool{\mathrm{ILT}_2(G)}+1$ if and only if $t \geq t_0$.
\section{Conclusion and further directions}

We introduced the cooling number of a graph, which quantifies the spread of a slow-moving contagion in a network. We gave tight bounds on the cooling number as functions of the order and diameter of the graph. Using isoperimetric techniques, we determined the cooling number of Cartesian grids. The cooling number of ILT graphs was considered in the previous section.

Several questions remain on the cooling number. Determining the exact value of the cooling number in various graph families, such as spiders and, more generally, trees, remains open. In the full version of the paper, we will consider the cooling number of other grids, such as strong or hexagonal grids. We want to classify the cooling number of ILT graphs where the initial graphs are not paths. Another direction to consider is the complexity of cooling, which is likely \textbf{NP}-hard.

\section{Acknowledgments}
Research supported by a grant of the first author from NSERC.

\end{document}